\documentclass[11pt]{article}
\usepackage[top=2.5cm, bottom=2.5cm, left=3.5cm, right=3.5cm]{geometry}

\usepackage{graphicx}
\usepackage{amsmath}
\usepackage{amssymb}
\usepackage{algorithm}
\usepackage{algorithmic}
\usepackage{overpic}
\usepackage{booktabs}
\usepackage{multirow}
\usepackage{array}
\usepackage{caption}
\usepackage{subcaption}
\usepackage{hyperref}
\usepackage{xcolor}
\usepackage{ifthen}
\usepackage{amsthm} 

\newtheorem{assumption}{Assumption}
\newtheorem{proposition}{Proposition}
\newtheorem{corollary}{Corollary}

\theoremstyle{definition}
\newtheorem{definition}{Definition}
\floatname{algorithm}{Algorithm}

\hypersetup{
    colorlinks=true,
    linkcolor=blue,
    citecolor=blue,
    urlcolor=blue
}

\newcommand{\safeincludegraphics}[2][]{%
    \IfFileExists{#2}{%
        \includegraphics[#1]{#2}%
    }{%
        \fbox{\parbox[c][0.24\textheight][c]{0.9\linewidth}{\centering Missing figure}}%
    }%
}

\newcommand{\Psit}{\Psi_{\theta}}
\newcommand{\Jpsi}{J_{\Psi}}

\title{Analysis of Error Propagation in Autoencoder-Based Reduced-Order Neural Ordinary Differential Equations}
\author{Jingyi Zhang and Gwanghyun Jo}
\date{}

\begin{document}

\maketitle

\begin{abstract}

Neural ODE reduced-order models often achieve comparable local prediction accuracy, yet their long-horizon extrapolation behavior can differ substantially. To analyze this discrepancy, we develop a path-integral identity that separates local discrepancy injection from amplification in the learned latent dynamics. The associated multi-step Jacobian norms quantify transport sensitivity and distinguish different propagation regimes.

Experiments on the Burgers and Gray--Scott systems exhibit two distinct patterns of error evolution. In Burgers systems, prediction errors remain bounded and are primarily associated with persistent local discrepancies. In contrast, Gray--Scott systems exhibit pronounced amplification during extrapolation, where Jacobian norms serve as sensitivity diagnostics rather than direct indicators of physical prediction accuracy.

\end{abstract}

\section{Introduction}
\label{sec:introduction}

Long-timescale simulation of nonlinear partial differential equations (PDEs) is a central task in many scientific and engineering applications, including fluid dynamics \cite{pope2000turbulent}, climate modeling, chemical reactions \cite{epstein1998introduction}, and engineering optimization \cite{nocedal2006numerical}. Because direct numerical simulation is often computationally expensive, reduced-order models (ROMs) seek to approximate the system evolution in a much lower-dimensional space while preserving the dominant dynamics \cite{benner2015survey,quarteroni2016reduced,hairer1993solving,hairer1996solving,ascher1998computer,butcher2016numerical}.

Classical ROM methods are largely based on linear projection techniques \cite{holmes2012turbulence,sirovich1987turbulence,berkooz1993proper}. Among them, Proper Orthogonal Decomposition (POD) provides energy-optimal spatial bases \cite{berkooz1993proper}, whereas Dynamic Mode Decomposition (DMD) focuses on temporal evolution \cite{schmid2010dmd,rowley2009spectral}. These approaches perform well for weakly nonlinear systems but become less effective when the underlying dynamics involve sharp gradients or complex nonlinear structures.

Recent learning-based ROMs address this limitation by introducing nonlinear latent representations through autoencoders \cite{hinton2006reducing,lee2020model,otto2019linearly,hesthaven2018non,fresca2021comprehensive,fresca2022pod}. Other developments include operator-learning frameworks such as DeepONet and Fourier Neural Operators \cite{lu2021learning,li2021fourier,kovachki2023neural}, together with physics-informed neural networks \cite{raissi2019physics,karniadakis2021physics}. Although these methods have substantially improved predictive capability, long-horizon performance is still assessed primarily through prediction errors \cite{vlachas2018data,pathak2018model,chattopadhyay2020long}. Such evaluations describe the final prediction quality but provide limited information about how errors evolve during autonomous rollout.

Neural ODEs \cite{chen2018neural} and their extensions \cite{kidger2020neural,qian2020lift,champion2019data,lusch2018deep} model latent dynamics as continuous-time differential equations. During recursive rollout, however, small local approximation errors accumulate and interact with the learned dynamics, so similar local prediction accuracy does not necessarily imply similar long-term behavior. Understanding how these errors propagate is therefore essential for evaluating Neural ODE reduced-order models.

This work focuses on the mechanisms governing rollout error propagation rather than prediction accuracy alone. During autonomous rollout, accumulated errors arise from two distinct processes: local discrepancies introduced at each prediction step and the transport of deviations that already exist in the latent trajectory. Local reconstruction metrics cannot distinguish between these two effects.

To analyze this behavior, we derive a path-integral identity that separates rollout error propagation into local discrepancy injection and transport amplification. Based on this decomposition, we introduce a Jacobian-based transport diagnostic that quantifies amplification along the learned latent dynamics. The resulting framework distinguishes weakly amplifying systems, where long-term behavior is governed primarily by local discrepancy injection, from strongly amplifying systems, where transport amplification becomes dominant. The proposed analysis is evaluated using 1D and 2D Burgers equations together with the Gray--Scott reaction--diffusion system, which together represent the two propagation regimes considered in this study.

The remainder of this paper is organized as follows. Section~2 introduces the reduced-order modeling framework and the boundary-initialized rollout protocol. Section~3 presents the theoretical analysis of rollout error propagation and transport amplification. Section~4 reports the experimental results for the Burgers and Gray--Scott systems. Finally, Section~5 concludes the paper.

\section{Method}
\label{sec:method}

This section presents the proposed autoencoder-based reduced-order Neural ODE framework. We first review the standard Neural Ordinary Differential Equation (Neural ODE) formulation \cite{chen2018neural} and its numerical realization through ODESolve. The formulation is then specialized to latent-space dynamics learned by the autoencoder, leading to the reduced-order prediction framework employed throughout this work.

\subsection{Neural Ordinary Differential Equations}
\label{sec:neural_ode}
In this subsection, we briefly review the standard Neural Ordinary
Differential Equation (Neural ODE) framework proposed by
Chen et al.~\cite{chen2018neural}. Unlike conventional discrete-time
models that directly predict the next state, Neural ODE parameterizes
the instantaneous vector field governing the continuous-time evolution.

Given a trajectory dataset
\[
\{h(t_0),h(t_1),\ldots,h(t_N)\},
\]
the objective of Neural ODE is to learn a neural network
$f_\theta$, whose induced continuous-time dynamics reproduce the observed trajectory.
The dynamics are modeled as the initial value problem
\begin{equation}
\frac{dh(t)}{dt}
=
f_\theta(h(t),t),
\qquad
h(t_0)=h_0,
\label{eq:standard_neural_ode}
\end{equation}
where $h(t)$ denotes the hidden state and
$f_\theta$ approximates the underlying continuous vector field.
Since Eq.~\eqref{eq:standard_neural_ode} generally admits no closed-form
solution, the state evolution is obtained by numerical integration. 
Throughout this work, numerical integration is implemented using the
fourth-order Runge--Kutta (RK4) method.
For completeness, we briefly summarize the RK4 scheme.
For each interval $[t_n,t_{n+1}]$, RK4 evaluates the learned vector
field at four intermediate stages,
\begin{equation}
\begin{aligned}
k_1
&=
f_\theta(h_n,t_n),\\
k_2
&=
f_\theta
\left(
h_n+\frac{\Delta t}{2}k_1,
t_n+\frac{\Delta t}{2}
\right),\\
k_3
&=
f_\theta
\left(
h_n+\frac{\Delta t}{2}k_2,
t_n+\frac{\Delta t}{2}
\right),\\
k_4
&=
f_\theta
\left(
h_n+\Delta t\,k_3,
t_n+\Delta t
\right),
\end{aligned}
\label{eq:rk4_stages}
\end{equation}
which produces the state update
\begin{equation}
\operatorname{RK4}(f_\theta,h_n,t_n,t_{n+1}):=
h_n
+
\frac{\Delta t}{6}
(k_1+2k_2+2k_3+k_4).
\label{eq:rk4_update}
\end{equation}
Whenever no ambiguity arises, $\operatorname{RK4}(f_\theta,h_n,t_n,t_{n+1})$
will simply be denoted by $\operatorname{RK4}(f_\theta,h_n)$.

During training, the parameters $\theta$ are optimized so that the
trajectory generated by repeated RK4 integration matches the observed
trajectory.
Because the RK4 integration consists entirely of differentiable
operations, gradients can be propagated through the integration
procedure using automatic differentiation, allowing the parameters of
the vector field $f_\theta$ to be optimized.
A common training objective is to minimize the discrepancy between the
predicted state and the one-step-ahead reference state over the training
interval, i.e.,
\begin{equation}
\mathcal{L}(\theta)
=
\sum_{n=0}^{N-1}
\left\|
\operatorname{RK4}(f_\theta,h_n)-h(t_{n+1})
\right\|_2^2.
\label{eq:general_loss}
\end{equation}
Different training objectives can be adopted to optimize the learned
vector field. In this work, a different training strategy is employed
to improve long-term rollout performance, as described in the following
subsection.

In the following subsection, the standard Neural ODE framework is
specialized to the latent space learned by an autoencoder, resulting in
the reduced-order Neural ODE model used throughout this work.


\subsection{Autoencoder Reduced-Order Neural ODE}
\label{AE_reduced_NODE}
Having reviewed the standard Neural ODE framework, we now specialize it
to a reduced-order setting based on nonlinear latent representations.
Instead of evolving the original high-dimensional state, the Neural ODE
is constructed in the latent space learned by an autoencoder. The latent
variable therefore becomes the dynamical state, while the same RK4
integration procedure introduced in the previous subsection is retained
for temporal evolution.

Let
\[
\{u_0,u_1,\ldots,u_N\}
\]
be a trajectory dataset sampled at uniformly spaced time instances
\[
t_n=t_0+n\Delta t,
\qquad
n=0,\ldots,N.
\]
The dataset is divided chronologically into the training and testing intervals,
\[
\mathcal T_{\mathrm{train}}
=
\{t_0,\ldots,t_{m}\},
\qquad
\mathcal T_{\mathrm{test}}
=
\{t_{m+1},\ldots,t_N\}.
\]

Rather than approximating vector field directly as in the previous Subsection, we first construct a nonlinear low-dimensional representation using an autoencoder. 
The encoder maps the physical state to the latent space, while the decoder reconstructs the physical state from its latent representation.

\subsubsection{Autoencoder-Based Latent Representation}

We introduce
\begin{equation}
z(t)
=
E_w(u(t)),
\qquad
\hat u(t)
=
D_w(z(t)),
\label{eq:autoencoder}
\end{equation}
where
$E_w:\mathbb{R}^d\rightarrow\mathbb{R}^k$
and
$D_w:\mathbb{R}^k\rightarrow\mathbb{R}^d$
denote the encoder and decoder parameterized by $w$, respectively.
Both the encoder and decoder are implemented as fully connected feedforward neural networks,
\begin{align*}
E_w(x)
&=
(L_3\circ\sigma\circ L_2\circ\sigma\circ L_1)(x),\\
D_w(x)
&=
(L_6\circ\sigma\circ L_5\circ\sigma\circ L_4)(x),
\end{align*}
where each $L_i$ denotes an affine linear transformation and $\sigma$ represents the nonlinear activation function.

The autoencoder is trained by minimizing the reconstruction loss
\[
\mathcal L_{\mathrm{AE}}
=
\sum_{i=0}^{N_{\mathrm{train}}}
\|
u_i
-
D_w(E_w(u_i))
\|_2^2.
\]
After training, the encoder produces the latent trajectory
\[
z_n
=
E_w(u_n),
\qquad
n=0,\ldots,N,
\]
which serves as the reference latent trajectory throughout this work.

\subsubsection{Latent neural ODE}
We next propose latent neural ODE method which learns the continuous-time dynamics in the latent space.
Once the autoencoder has been trained, we obtain the encoded latent trajectory
$\{z_n\}_{n=0}^{N}$
We model the latent trajectories by the autonomous continuous-time system
\begin{equation}
\frac{dz}{dt}
=
g_\theta(z),
\label{eq:latent_ode}
\end{equation}
where $g_\theta$ denotes the latent vector field to be learned. 
Consistent with Section~\ref{sec:neural_ode}, the neural network parametrizes only the instantaneous latent derivative rather than directly predicting future latent states. Temporal evolution is recovered by $\operatorname{RK4}$ as in (\ref{eq:rk4_update}).

State is propagated over each observation interval using the RK4
integration introduced in Section~\ref{sec:neural_ode}. This naturally
defines the one-step latent flow operator
\begin{equation}
\Psi_\theta(z_n):=\operatorname{RK4}(g_\theta,z_n),
\label{eq:rk4}
\end{equation}
Repeated applications of $\Psi_\theta$ define the $k$-step flow operator
$\Psi_\theta^{(k)}$, where
\[
\Psi_\theta^{(k)}
=
\underbrace{\Psi_\theta\circ\Psi_\theta\circ\cdots\circ\Psi_\theta}_{k\ \text{times}},
\]
which is used to generate $k$-step-ahead predictions.
The latent vector field $g_\theta$ is learned using only the training
interval. The proposed training strategy will be described in the
following subsection. 

After training, future latent states are generated recursively from the
last training state by
\[
\Psi_\theta^{(k)}(z_m),
\qquad
k=1,2,\ldots.
\]
The corresponding test error is computed as
\[
\sum_{k=1}^{N-m}
\left\|
\Psi_\theta^{(k)}(z_m)-z_{m+k}
\right\|_2^2.
\]
The resulting rollout trajectory forms the basis of the error
propagation analysis developed in
Section~\ref{sec:framework}. Finally, the predicted physical state is
reconstructed by
$D_w\!\left(\Psi_\theta^{(k)}(z_m)\right).$
Algorithm~\ref{alg:ae_node} summarizes the complete training and prediction procedure adopted throughout this work. 

\begin{algorithm}[!ht]
\caption{Autoencoder Reduced-Order Neural ODE}
\label{alg:ae_node}
\begin{algorithmic}[1]
\REQUIRE
Trajectory dataset $\{u_n\}_{n=0}^{N}$,
training boundary $m$
\ENSURE
Learned latent vector field $g_\theta$ and predicted trajectory
$\{\hat u_n\}_{n=m+1}^{N}$

\STATE Train the encoder--decoder pair $(E_w,D_w)$ by minimizing
$\mathcal L_{\mathrm{AE}}$.

\STATE Encode the trajectory to obtain the latent states
\[
z_n=E_w(u_n),
\qquad n=0,\ldots,N.
\]

\STATE Learn the latent vector field $g_\theta$ using the training
interval.

\STATE Construct the one-step latent flow operator
\[
\Psi_\theta(z)=\operatorname{RK4}(g_\theta,z).
\]

\FOR{$k=1,\ldots,N-m$}
    \STATE Predict the latent state
    \[
    z_{m+k}
    =
    \Psi_\theta^{(k)}(z_m).
    \]
    \STATE Reconstruct the physical state
    \[
    \hat u_{m+k}
    =
    D_w(z_{m+k}).
    \]
\ENDFOR

\RETURN
$g_\theta$ and
$\{\hat u_n\}_{n=m+1}^{N}$.
\end{algorithmic}
\end{algorithm}


\subsection{Training Strategy}
\label{training_strategy}

The training proceeds in two stages: first, the autoencoder learns a latent representation by minimizing $\mathcal L_{\mathrm{AE}}$; subsequently, the latent dynamics are optimized through multi-step rollout prediction. This allows the learned flow map to account for error accumulation during recursive forecasting.

Valid rollout initialization points satisfy
$n \in \{0,\ldots,N_{\rm train}-k_{\max}\}$,
so that all rollout targets remain within the training interval. The latent dynamics are trained by minimizing

\begin{equation}
\mathcal L_{\rm dyn}
=
\frac{1}{|\mathcal I_{\rm train}'|}
\sum_{n\in\mathcal I_{\rm train}'}
\sum_{k=1}^{k_{\max}}
p(k)
\left\|
\Psi_\theta^{(k)}(z_n)
-
z_{n+k}
\right\|_2^2,
\label{eq:dyn_loss}
\end{equation}

where
\[
\sum_{k=1}^{k_{\max}} p(k)=1.
\]

The weighting function $p(k)$ determines the rollout exposure strategy. Across the rollout-training variants, the autoencoder architecture, latent Neural ODE, optimizer, learning rate, and training data are kept unchanged.

Three rollout exposure strategies are considered:

\begin{enumerate}
    \item \textbf{Base.}
    Training uses only the maximum rollout horizon,
    \[
    p(k)=
    \begin{cases}
    1, & k=k_{\max},\\
    0, & \text{otherwise}.
    \end{cases}
    \]

    \item \textbf{Mixed.}
    All rollout horizons are weighted uniformly,
    \[
    p(k)=\frac{1}{k_{\max}},
    \qquad
    k=1,\ldots,k_{\max}.
    \]

    \item \textbf{Mix2.}
    Longer rollout horizons receive progressively larger weights,
    \[
    p(k)=
    \frac{k^\alpha}
    {\sum_{j=1}^{k_{\max}} j^\alpha},
    \qquad
    \alpha=1.
    \]
\end{enumerate}

In addition to these rollout exposure strategies, we evaluate a
\textbf{Gradient Clipping} optimization baseline. This baseline does not alter the rollout objective in Eq.~(\ref{eq:dyn_loss}) or the weighting distribution $p(k)$. Instead, after backpropagation and before each optimizer update, the global $\ell_2$ norm of the parameter gradients is clipped according to

\[
\nabla_\theta \mathcal L_{\rm dyn}
\leftarrow
\min\left(
1,
\frac{\tau_{\rm clip}}
{\left\|\nabla_\theta \mathcal L_{\rm dyn}\right\|_2}
\right)
\nabla_\theta \mathcal L_{\rm dyn}.
\]

In all reported Gradient Clipping experiments, the clipping threshold is set to
\[
\tau_{\rm clip}=1.0.
\]

Gradient Clipping therefore modifies only the optimization dynamics, whereas Base, Mixed, and Mix2 modify the distribution of rollout horizons used in the training objective.

Unlike conventional Neural ODE studies that evaluate models primarily through rollout prediction errors, the present framework is designed to analyze error propagation mechanisms. The resulting rollout trajectories, latent deviations, and learned flow maps provide the basis for the path-integral decomposition and Jacobian transport diagnostics developed in the following section.


\section{Error Propagation and Diagnostic Framework}
\label{sec:framework}
The latent rollout generated from the last training state is
\[
\Psi_\theta^{(k)}(z_m),
\qquad
k=1,2,\ldots.
\]
Here, $z_{m+k}$ denotes the encoded latent state obtained from the
reference trajectory, whereas
$\Psi_\theta^{(k)}(z_m)$ represents the autonomous rollout prediction.

\subsection{Prediction Error Formulation}
\label{sec:rollout_prediction}
In this subsection, we introduce the error measures used throughout the
remainder of this work.
The latent prediction error measures the deviation between the
autonomous rollout and the encoded reference trajectory:
\[
\epsilon_k
=
\Psi_\theta^{(k)}(z_m)-z_{m+k}.
\]
To assess the quality of the learned latent representation,
\begin{equation}
E_{\mathrm{rep}}(t_n)
=
\frac{
\| D_w(E_w(u_n))-u_n \|_2
}{
\| u_n \|_2 + \eta_{\mathrm{reg}}
}, \quad n=0,\ldots, N,
\label{eq:rep_error}
\end{equation}
where $\eta_{\mathrm{reg}} = 10^{-8}$ ensures numerical stability.
Next, we define the physical-space prediction error over the test interval:
\begin{equation}
E_{\mathrm{pred}}(t_{m+k})
=
\frac{ \|  D_w(\Psi_\theta^{(k)}(z_m))  -u_{m+k} \|_2}{\| u_{m+k} \|_2 + \eta_{\mathrm{reg}}}, \quad k=1,\ldots N-m.
\label{eq:pred_error}
\end{equation}
Finally, we define the one-step local discrepancy
\begin{align}
\delta_k = \Psit(z_{m+k})-z_{m+k+1}.
\end{align}
Using the above definitions, it follows that
\begin{equation}
\epsilon_{k+1}
=
\Bigl[
\Psi_\theta^{(k+1)}(z_m)
-
\Psi_\theta(z_{m+k})
\Bigr]
+\delta_k.
\label{eq:latent_error_split}
\end{equation}
The first term represents the accumulated tracking error propagated
through the learned dynamics, whereas the second term corresponds to
the one-step local discrepancy introduced at the current prediction
step.
This decomposition separates transport from local discrepancy injection.
Fig.~\ref{fig:loss_concept} illustrates the relationship among the
reference latent trajectory, autonomous rollout, latent prediction
error, and one-step local discrepancy.

\begin{figure}[t]
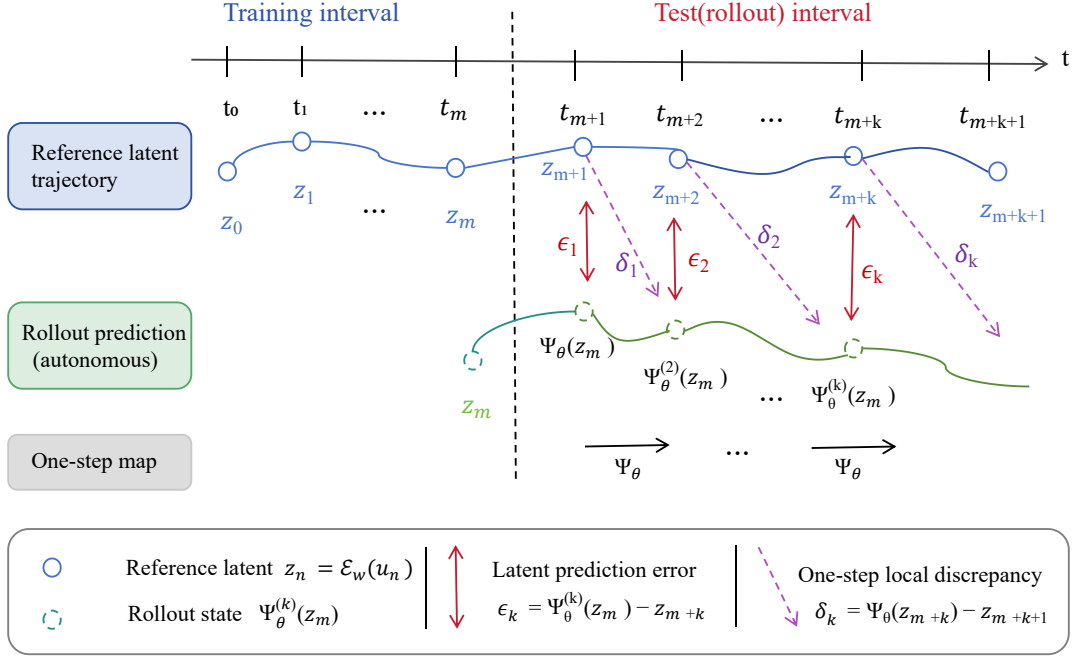

    \centering
    \safeincludegraphics[width=0.98\textwidth]{Figure_concept.pdf}
    \caption{
    Schematic overview of the rollout process and the error quantities used throughout the paper.
    Starting from the boundary latent state $z_m$, the learned flow map recursively generates the rollout trajectory
    $\Psi_{\theta}^{(k)}(z_m)$,
    while the encoded latent trajectory provides the reference states $z_{m+k}$.
    The latent prediction error
    $\epsilon_n=\Psi_{\theta}^{(n-m)}(z_m)-z_n$
    measures the accumulated deviation from the reference trajectory, whereas the one-step local discrepancy
    $\delta_n=\Psi_{\theta}(z_n)-z_{n+1}$
    represents the local prediction error introduced at each rollout step.
    The bottom panel illustrates the decomposition of the tracking error into transport and local discrepancy injection, which motivates the theoretical framework developed in Section~3.
    }
    \label{fig:loss_concept}
\end{figure}

\subsection{Jacobian Transport Diagnostic}
\label{sec:jacobian_transport_diagnostic}

The decomposition introduced above separates the transported deviation from the one-step local discrepancy. To quantify how the learned rollout map transports a local perturbation over multiple steps, we evaluate products of its Jacobian matrices along the encoded reference trajectory.

\begin{definition}[Sequential Jacobian norm]
\label{def:transport_amplification}
The ordered product of sequential Jacobian matrices evaluated along the encoded reference trajectory is defined as
\[
P_{n,k} = \Jpsi(z_{n+k-1})\cdots\Jpsi(z_n), \qquad S_{n,k} = \|P_{n,k}\|_2,
\]
where
\[
\Jpsi(z):=D\Psi_\theta(z)
\]
denotes the Jacobian of the learned rollout map, and $P_{n,k}$ denotes the corresponding ordered Jacobian product over $k$ rollout steps. The matrix multiplication order places the Jacobian at the latest time on the left. The scalar $S_{n,k}$ is the \emph{multi-step spectral norm}, which quantifies the local linearized transport amplification over $k$ steps along the reference trajectory.
\end{definition}

The product $P_{n,k}$ is the Jacobian chain product of the repeated rollout map derivative evaluated along the encoded reference trajectory. It is a trajectory-based linearized diagnostic and does not capture the exact nonlinear error evolution along the perturbed rollout path. When $S_{n,k}$ remains of order one, the learned map does not produce strong multi-step linear amplification along that reference segment; we refer to this as the weak amplification regime, where repeated discrepancy injection may remain an important contributor to the overall error. When $S_{n,k}\gg1$, the map possesses the capacity to amplify deviations aligned with sensitive directions, although the realized amplification also depends on the alignment of the deviation with those directions. This case is referred to as the strong amplification regime.

Algorithm~\ref{alg:jacobian_computation} summarizes the computation of $P_{n,k}$ and $S_{n,k}$ over a prescribed horizon.

\begin{algorithm}[!ht]
\caption{Computation of the Multi-step Jacobian Norm}
\label{alg:jacobian_computation}
\begin{algorithmic}[1]
\REQUIRE Learned one-step rollout map $\Psit$, reference latent trajectory $\{z_0,\ldots,z_N\}$, rollout horizon $K$
\ENSURE Multi-step Jacobian norms $\{S_{n,1},\ldots,S_{n,K}\}$
\STATE Initialize $P_{n,0}\gets I$
\FOR{$k=1$ \textbf{to} $K$}
    \STATE Compute the local derivative matrix $J_{n+k-1}=\Jpsi(z_{n+k-1})$ using automatic differentiation.
    \STATE Update the ordered Jacobian product: $P_{n,k}=J_{n+k-1}P_{n,k-1}$.
    \STATE Compute the spectral norm $S_{n,k}=\|P_{n,k}\|_2$ via singular value decomposition.
\ENDFOR
\STATE \RETURN $\{S_{n,1},\ldots,S_{n,K}\}$.
\end{algorithmic}
\smallskip
\noindent
\textbf{Note:} For high-dimensional latent manifolds where full matrix storage becomes prohibitive, the spectral norm can be estimated efficiently using matrix-free Jacobian-vector products combined with power iteration or a Lanczos method.
\end{algorithm}

The sequence $\{S_{n,k}\}$ is used as a numerical diagnostic of linearized transport. The next subsection relates this diagnostic to the exact and linearized evolution of the latent rollout deviation.

\subsection{Error Propagation Analysis}
\label{sec:error_propagation_analysis}

We now analyze the evolution of the latent rollout deviation $\epsilon_n$ and the one-step local discrepancy $\delta_n$ introduced in Section~\ref{sec:rollout_prediction}. For consistency with the notation established in Section~2, the rollout prediction is expressed directly through repeated applications of the learned rollout map $\Psi_\theta$, without introducing additional rollout-state notation.

For the recursive analysis below, it is convenient to rewrite the rollout deviation using the absolute time index
\[
n=m+k.
\]
Accordingly, the latent rollout deviation is defined as
\[
\epsilon_n
=
\Psi_\theta^{(n-m)}(z_m)
-
z_n,
\qquad
n\ge m,
\]
where $z_n$ denotes the encoded reference trajectory and $\Psi_\theta^{(n-m)}(z_m)$ denotes the autonomous rollout prediction initialized from the last training latent state $z_m$. The analysis requires that both the encoded reference trajectory and the autonomous rollout prediction remain within a compact domain on which $\Psi_\theta$ possesses the required differentiability.

By the definition of the rollout deviation,
\[
\Psi_\theta^{(n-m)}(z_m)
=
z_n+\epsilon_n.
\]
Applying the learned rollout map once more gives
\[
\Psi_\theta^{(n+1-m)}(z_m)
=
\Psi_\theta(z_n+\epsilon_n).
\]
Using the definition of the one-step local discrepancy
\[
\delta_n
=
\Psi_\theta(z_n)-z_{n+1},
\]
the one-step error recursion becomes
\[
\epsilon_{n+1}
=
\Psi_\theta(z_n+\epsilon_n)
-
\Psi_\theta(z_n)
+
\delta_n,
\]
which expresses the next deviation as the transported current deviation together with the local discrepancy injected at the current reference transition. This identity serves as the starting point for the path-integral representation developed below.

\begin{assumption}[Regularity of the learned rollout map]
\label{ass:regularity}
The reference latent trajectory
\[
z_n,
\]
the autonomous rollout prediction
\[
\Psi_\theta^{(n-m)}(z_m),
\]
and the interpolation path
\[
z_n+s\epsilon_n,\qquad s\in[0,1],
\]
remain in a compact subset
\[
\Omega\subset\mathbb R^d
\]
over the evaluation horizon. The learned rollout map $\Psi_\theta$ is continuously differentiable on an open neighborhood of $\Omega$. Whenever a second-order Taylor expansion is used, $\Psi_\theta$ is assumed to be twice continuously differentiable on that neighborhood.
\end{assumption}

Under these conditions, the one-step local discrepancy and the transported deviation can be related exactly through the error recursion.

\begin{proposition}[Exact path-integral identity]
\label{prop:path_integral}
Under the conditions specified in Assumption~\ref{ass:regularity}, the exact evolution of the latent rollout deviation satisfies the relationship
\begin{equation}
\epsilon_{n+1}
=
\left( \int_0^1 \Jpsi(z_n+s\epsilon_n)\,ds \right) \epsilon_n + \delta_n,
\label{eq:error_transport}
\end{equation}
The first term gives an exact path-integral representation of the transport of the current latent deviation under the learned rollout map, whereas the second term represents the one-step local discrepancy injected at the current reference transition.
\end{proposition}

\begin{proof}
Define $g_n(s) := \Psit(z_n+s\epsilon_n)$ for $s\in[0,1]$, which parameterizes the line segment connecting the reference latent trajectory and the corresponding autonomous rollout prediction. Then
\[
g_n(1)-g_n(0)
=
\Psit(z_n+\epsilon_n)-\Psit(z_n).
\]
The fundamental theorem of calculus gives $g_n(1)-g_n(0)=\int_0^1 g_n'(s)\,ds$. Differentiating with respect to $s$ yields $g_n'(s)=\Jpsi(z_n+s\epsilon_n)\epsilon_n$. Combining these relations with the one-step recursion
\[
\epsilon_{n+1}
=
\Psit(z_n+\epsilon_n)-\Psit(z_n)+\delta_n
\]
completes the proof.
\end{proof}

For a conservative norm bound on the one-step deviation, we bound the path-integral operator by a uniform Jacobian norm bound over the domain.

\begin{corollary}[Worst-case single-step bound]
\label{cor:norm_bound}
Under Assumption~\ref{ass:regularity}, the magnitude of the latent deviation is bounded by
\begin{equation}
\|\epsilon_{n+1}\| \le L_\Psi \, \|\epsilon_n\| + \|\delta_n\|,
\label{eq:norm_bound}
\end{equation}
where $L_\Psi := \sup_{z\in\Omega}\|\Jpsi(z)\|$ provides a Lipschitz bound for $\Psi_\theta$ on the domain $\Omega$.
\end{corollary}

Corollary~\ref{cor:norm_bound} provides a bound for one rollout step. Long-horizon analysis requires a recursive expansion that separates the transported initial deviation, the transported historical discrepancies, and the nonlinear remainder.

\begin{proposition}[Linearized multi-step propagation with remainder]
\label{prop:multistep_linear}
Suppose Assumption~\ref{ass:regularity} holds and let the learned map $\Psi_\theta$ be twice continuously differentiable. The latent deviation accumulated over a horizon of $k$ steps admits the decomposition
\begin{equation}
\epsilon_{n+k}
=
P_{n,k}\,\epsilon_n + \sum_{j=0}^{k-1} P_{n+j+1,\,k-j-1}\,\delta_{n+j} + R_{n,k},
\label{eq:multistep_linear}
\end{equation}
where the linear operators are defined by the ordered Jacobian products
\[
P_{n,k}
=
\Jpsi(z_{n+k-1})\cdots\Jpsi(z_n),
\qquad
P_{n+j+1,\,k-j-1}
=
\Jpsi(z_{n+k-1})\cdots\Jpsi(z_{n+j+1}),
\]
with the convention $P_{q,0}=I$ for any $q$. The nonlinear remainder $R_{n,k}$ collects the Taylor remainders generated at each step after their transport through the subsequent Jacobian factors, and satisfies $\|R_{n,k}\| \le C_k \sum_{\ell=0}^{k-1} \|\epsilon_{n+\ell}\|^2$, where the constant $C_k$ depends on the derivative bounds of $\Psi_\theta$ and on the rollout horizon $k$. This decomposition is a local linearization around the encoded reference trajectory; when the deviation becomes large or the Jacobian products exhibit strong amplification, the nonlinear remainder may no longer be negligible relative to the linearized terms.
\end{proposition}

\begin{proof}
Performing a one-step Taylor expansion of the learned map about the reference state gives
\[
\Psit(z_n+\epsilon_n) = \Psit(z_n) + \Jpsi(z_n)\epsilon_n + r_n,
\]
where the remainder satisfies $\|r_n\|\le C\|\epsilon_n\|^2$. Substituting this expansion into the discrete error evolution equation yields
\[
\epsilon_{n+1} = \Jpsi(z_n)\epsilon_n + \delta_n + r_n.
\]
Applying this relation recursively over $k$ steps produces the multi-step expression in Eq.~\eqref{eq:multistep_linear}. The initial deviation $\epsilon_n$ is transported by the full chain of Jacobian products. Each one-step local discrepancy $\delta_{n+j}$ is transported by the Jacobian product over the remaining $k-j-1$ steps. The Taylor remainder generated at each step is transported by the Jacobian factors associated with the subsequent steps. Collecting these transported remainder terms gives $R_{n,k}$ and the stated horizon-dependent bound.
\end{proof}

\paragraph{Significance of the Linearized Diagnostic.}
The sequential Jacobian product introduced in Definition~\ref{def:transport_amplification} serves as a diagnostic for mapping local linearized transport sensitivities. Grouping the historical discrepancy terms defines the cumulative transported discrepancy
\[
\Xi_{n,k} := \sum_{j=0}^{k-1} P_{n+j+1,\,k-j-1}\,\delta_{n+j},
\]
which represents the cumulative contribution of transported one-step local discrepancies under the linearized model. This yields the compact decomposition
\begin{equation}
\epsilon_{n+k} = P_{n,k}\,\epsilon_n + \Xi_{n,k} + R_{n,k}.
\label{eq:compact_decomp}
\end{equation}
Here $P_{n,k}\epsilon_n$ is the transported initial deviation, $\Xi_{n,k}$ is the cumulative transported discrepancy, and $R_{n,k}$ is the nonlinear correction. When $S_{n,k}$ is of order one, repeated discrepancy injection may still accumulate; when $S_{n,k}\gg1$, the learned rollout map has the capacity to strongly amplify deviations aligned with sensitive directions. The realized error at a given horizon reflects the combined effect of the transported initial deviation, the accumulated discrepancy, and nonlinear corrections beyond the local linearization.

The linearized decomposition is local. A separate conservative envelope follows by applying the single-step norm bound recursively.

\begin{proposition}[Global rollout deviation bound]
\label{prop:global_bound}
Assuming the learned model satisfies uniform bounds such that $\|\delta_n\| \le \bar\delta$ and $\|\Jpsi(z)\| \le L_\Psi$ hold across the domain $\Omega$, the latent rollout deviation is bounded by
\begin{equation}
\|\epsilon_{n+k}\| \le L_\Psi^k \,\|\epsilon_n\| + \bar\delta \sum_{j=0}^{k-1} L_\Psi^{\,j}.
\label{eq:global_bound_ineq}
\end{equation}
For $L_\Psi \neq 1$, this geometric series collapses to the closed expression
\begin{equation}
\|\epsilon_{n+k}\| \le \bar\delta\,\frac{L_\Psi^{\,k} - 1}{L_\Psi - 1} + L_\Psi^k \,\|\epsilon_n\|.
\label{eq:global_bound_closed}
\end{equation}
For $L_\Psi = 1$, the bound reduces to
\begin{equation}
\|\epsilon_{n+k}\| \le \|\epsilon_n\| + k\bar\delta.
\label{eq:global_bound_linear}
\end{equation}
\end{proposition}

The bound based on $L_\Psi$ provides a conservative domain-wide envelope, whereas $P_{n,k}$ and $S_{n,k}$ retain the trajectory-dependent structure of the learned rollout map. The numerical experiments in Section~4 examine these quantities together with the rollout error and one-step local discrepancy to explain the different propagation behaviors observed among the benchmark systems and training strategies.

\section{Results}
\label{sec:results}

This section addresses whether long-term rollout error primarily arises from local discrepancy injection or from transport amplification. Section~\ref{sec:framework} developed a theoretical framework that distinguishes these two mechanisms: injection-dominated propagation occurs when the latent Jacobian $S_{n,k}$ remains of order one, whereas amplification-dominated propagation arises when $S_{n,k}$ becomes substantially larger than one. Across physical systems with contrasting stability properties, the experiments that follow test this prediction.

Burgers equations (1D and 2D) serve as representative cases of mild transport sensitivity, where the latent dynamics are expected to remain in a regime of subdued error growth. Gray--Scott reaction--diffusion, by contrast, serves as a strong-amplification case, where the latent Jacobian is predicted to enter a regime of exponential error growth. Whether the theoretical classification in Section~\ref{sec:framework} correctly separates stable from unstable rollout behavior is thus tested in a controlled manner.

Implementation of the framework uses PyTorch v1.12.1 with CUDA 11.3 acceleration, in a modular design that separates autoencoder, latent ODE dynamics, training pipelines, and diagnostics. Network architectures and optimization protocols remain fixed throughout, with only input and output dimensions adjusted for each physical system. All experiments were executed on a computing cluster with NVIDIA V100 GPUs (32GB memory), 20-core Intel Xeon CPUs, and 256GB RAM per node; model definitions, training routines, rollout scripts, and Jacobian diagnostic tools are provided in the codebase for full reproducibility.

\begin{table}[!t]
\centering
\caption{
Training hyperparameters and implementation settings.
}
\label{tab:implementation}
\begin{tabular}{llll}
\toprule
Parameter & 1D Burgers & 2D Burgers & Gray--Scott \\
\midrule
Latent dimension & 16 & 16 & 16 \\
Hidden width & 128 & 128 & 128 \\
Number of layers & 3 & 3 & 3 \\
Activation & ReLU & ReLU & ReLU \\
Optimizer & Adam & Adam & Adam \\
Learning rate & $10^{-3}$ & $10^{-3}$ & $10^{-3}$ \\
Batch size & 64 & 64 & 64 \\
Training epochs & 2000 & 2000 & 2000 \\
\bottomrule
\end{tabular}
\end{table}

\subsection{Burgers Systems: Weak Amplification Regime}
\label{sec:burgers_systems}

Throughout the prediction horizon, the 1D Burgers rollout remains stable, with predicted solutions staying visually close to the ground truth (Fig.~\ref{fig:1d_solution_snapshots}). Relative $L^2$ errors remain below 2\%, and the local discrepancy injection $\|\delta_n\|$ is approximately 0.01. The bounded growth of $\|\epsilon_n\|$ across seeds is consistent with propagation dominated by local discrepancy injection rather than by transport-driven escalation (Fig.~\ref{fig:1d_mechanism}).

Consistent with the transport sensitivity diagnostic $S_{n,k} \approx 1$, log-slopes of $\|\epsilon_n\|$ cluster near zero across all models, and test-horizon relative $L^2$ errors remain near 1.5\% for both predictive models (Table~\ref{tab:1d_combined}).

\begin{figure}[t]
\centering
\safeincludegraphics[width=\textwidth]{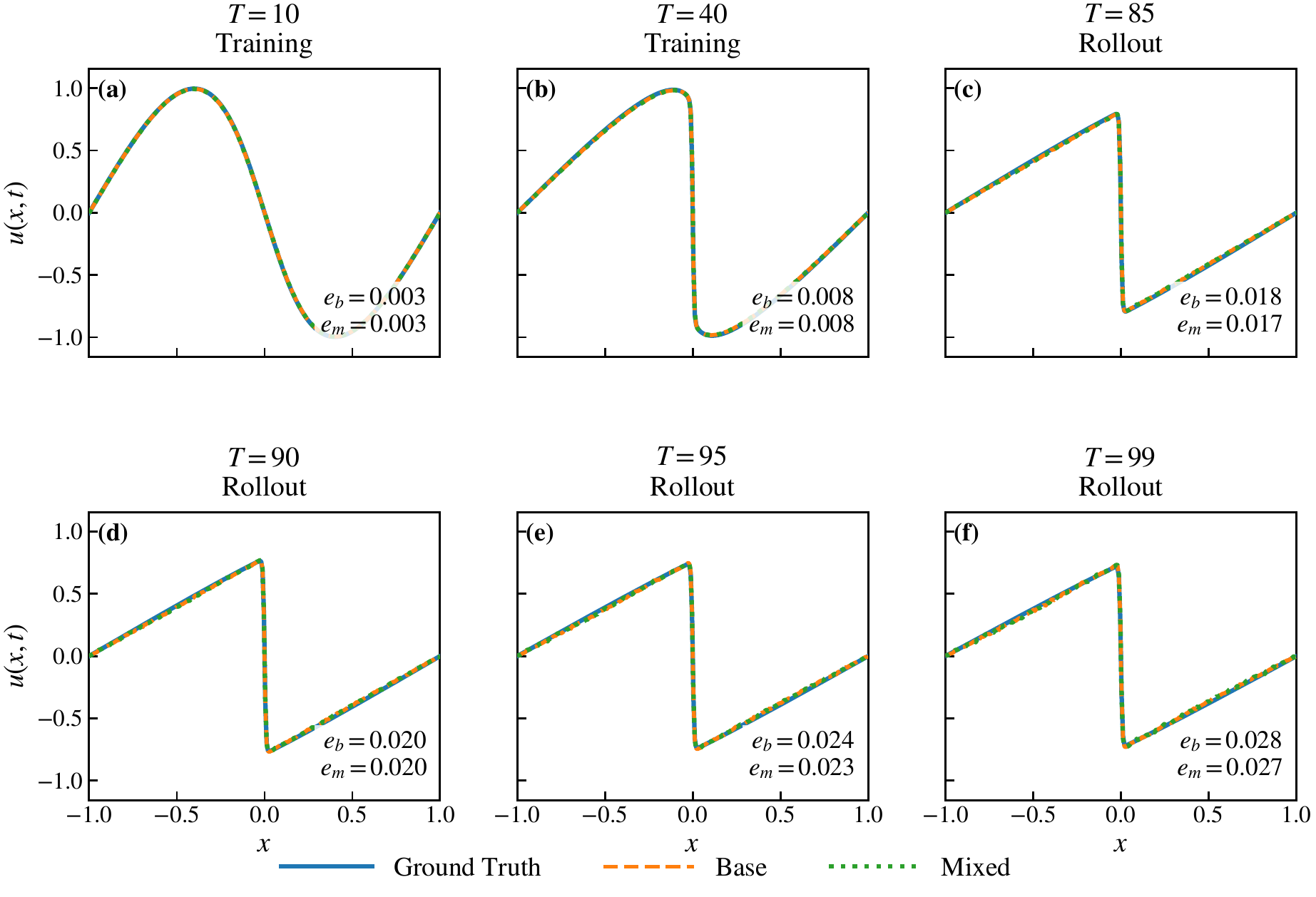}
\caption{
Representative 1D Burgers solution snapshots under the boundary-initialized evaluation protocol.
Ground-truth, Base, and Mixed solutions are overlaid for visual comparison.
Throughout rollout the curves remain visually close, reflecting the mild transport sensitivity of the Burgers system.
The reported relative $\ell_2$ errors indicate the gradual accumulation of small prediction discrepancies.
}
\label{fig:1d_solution_snapshots}
\end{figure}

In this regime of limited transport sensitivity, the two rollout strategies produce nearly identical behavior. Their empirical log-slopes are statistically indistinguishable, and modifying rollout exposure has negligible impact on the final error (Table~\ref{tab:1d_combined}).

\begin{figure}[t]
\centering
\safeincludegraphics[width=\textwidth]{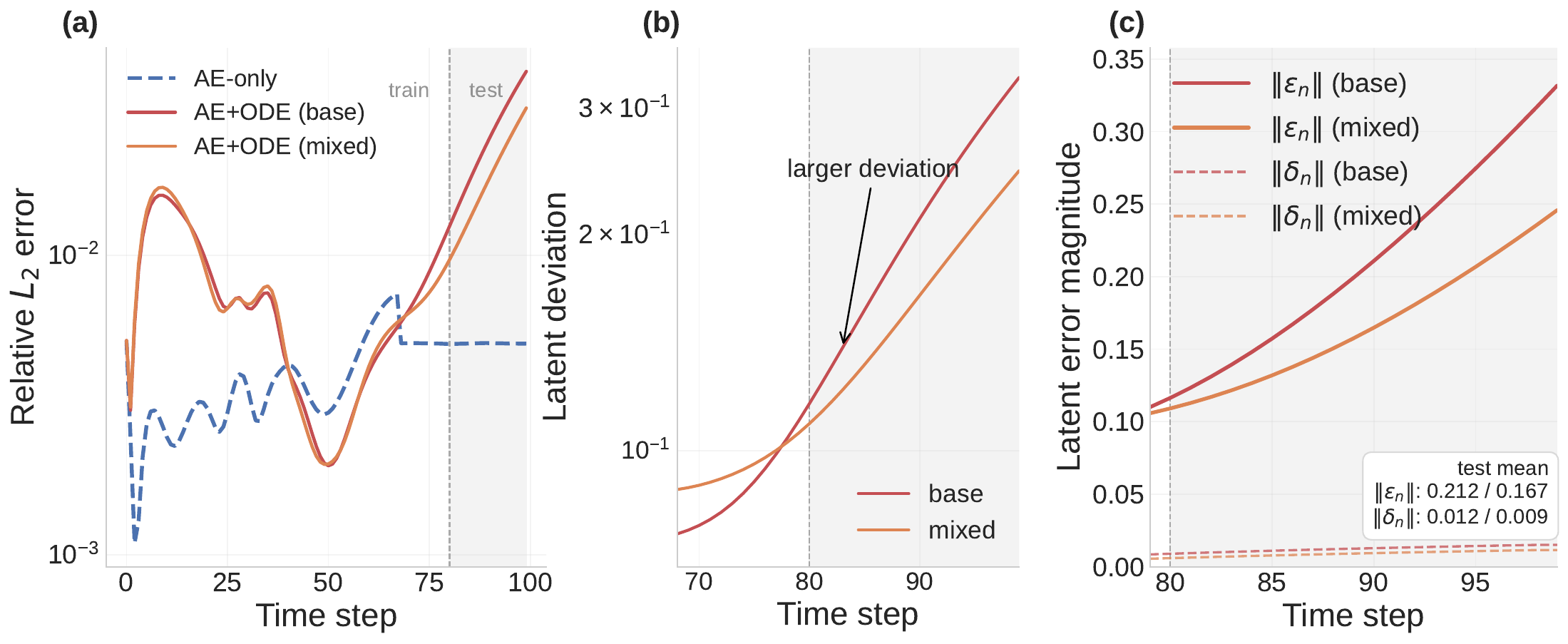}
\caption{
1D Burgers mechanism diagnostics showing error evolution, local discrepancy injection, and latent rollout deviation growth across seeds.
Bounded growth of $\|\epsilon_n\|$ indicates that error propagation is dominated by local discrepancy injection rather than by transport-driven escalation.
}
\label{fig:1d_mechanism}
\end{figure}

\renewcommand{\arraystretch}{1.1}

\begin{table}[t]
\centering
\caption{
1D Burgers propagation diagnostics (10 seeds): 
Mean $\pm$ standard deviation for physical-space relative $L^2$ error (RelL2),
local discrepancy injection ($\|\delta_n\|$), and empirical growth rate of $\|\epsilon_n\|$ (log-slope).
}
\label{tab:1d_combined}
\small
\setlength{\tabcolsep}{5pt}
\begin{tabular*}{\linewidth}{@{\extracolsep{\fill}} l c c c @{}}
\toprule
Method & Test RelL2 (\%) & Mean $\|\delta_n\|$ & Log-slope of $\|\epsilon_n\|$ \\
\midrule
\multicolumn{4}{l}{\textit{Representation (no rollout)}} \\
AE-only & $1.890 \pm 0.714$ & -- & -- \\
\midrule
\multicolumn{4}{l}{\textit{Predictive dynamics (with rollout)}} \\
AE+ODE (base)  & $1.549 \pm 0.487$ & 0.012 & $0.054 \pm 0.029$ \\
AE+ODE (mixed) & $1.419 \pm 0.377$ & 0.009 & $0.055 \pm 0.031$ \\
\bottomrule
\end{tabular*}
\end{table}

Close agreement with the encoded reference trajectory is also observed for the 2D Burgers velocity fields, with spatial relative errors staying below 1.2\% across all three exposure strategies (Fig.~\ref{fig:2d_solution_snapshots}). Localized front structures are preserved throughout the rollout horizon.

\begin{figure}[t]
\centering
\safeincludegraphics[width=\textwidth]{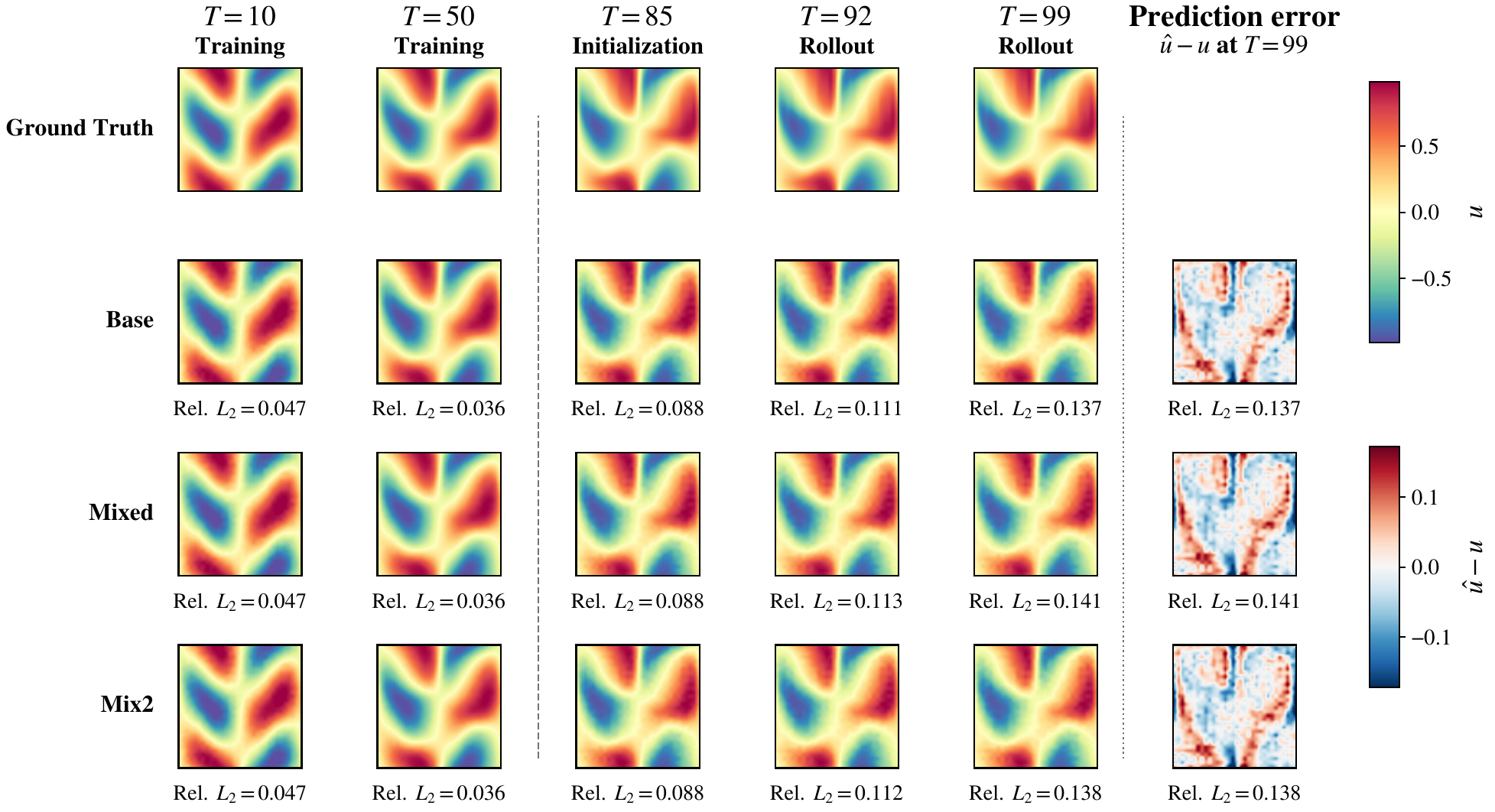}
\caption{
Representative 2D Burgers solution snapshots under the boundary-initialized evaluation protocol.
The first two columns correspond to autoencoder reconstructions within the training interval, the third column denotes the latent state used for boundary initialization, and the following columns show latent ODE rollout predictions in the test interval.
Rows compare the ground truth with the Base, Mixed, and Mix2 models.
The rightmost column shows the signed prediction error $\hat{u}-u$ at the final rollout time.
Numbers below each predicted snapshot indicate the relative $\ell_2$ error.
}
\label{fig:2d_solution_snapshots}
\end{figure}

Across all exposure variants, rollout errors remain within 1\% of the AE-only baseline (Fig.~\ref{fig:2d_main}; Table~\ref{tab:2d_perf}). Temporal correlation $H_\tau$ and log-slope values exhibit minimal variation across loss formulations, as expected under the constrained transport sensitivity of this regime.

\begin{figure}[t]
\centering
\safeincludegraphics[width=\linewidth]{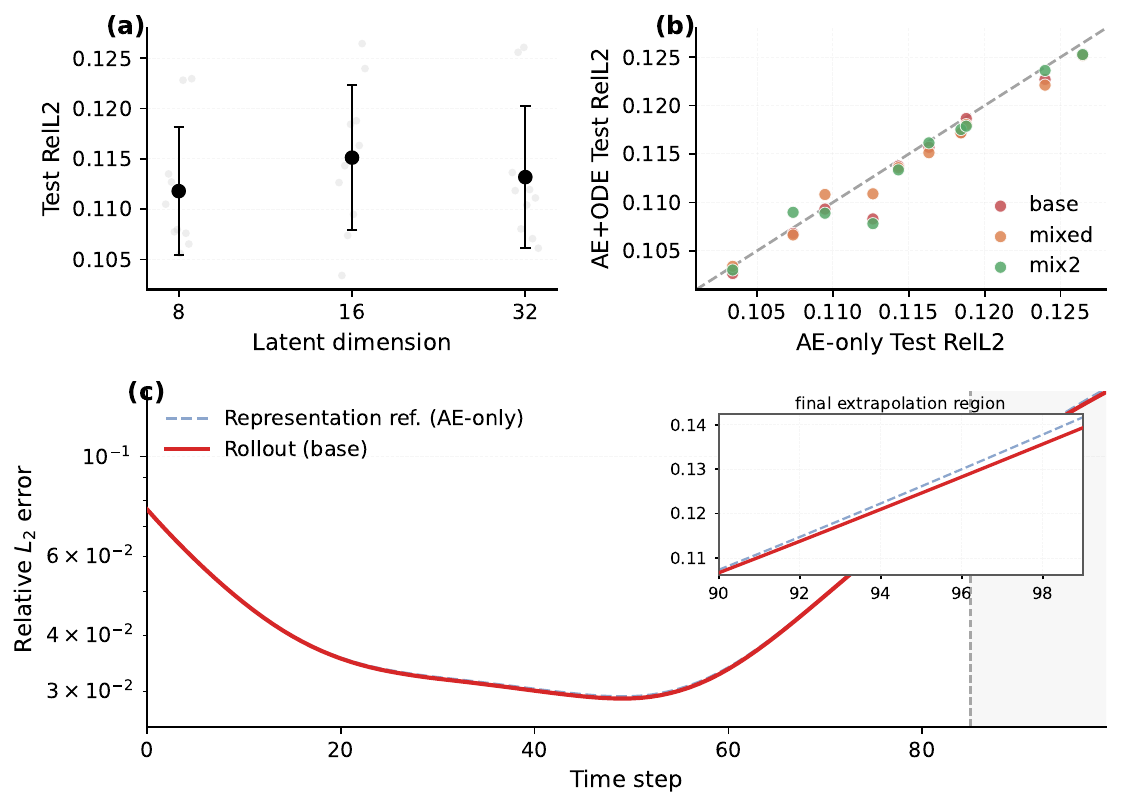}
\caption{
2D Burgers rollout behavior showing consistent performance across exposure variants, with errors remaining within approximately $1\%$ of the AE-only baseline.
}
\label{fig:2d_main}
\end{figure}

\begin{table}[t]
\centering
\caption{
2D Burgers propagation diagnostics (10 seeds): 
Comparison of representation error (AE-only) with rollout behavior (AE+ODE).
}
\label{tab:2d_perf}
\begin{tabular}{lcccc} 
\toprule
Method & Train RelL2 & Test RelL2 & $H_\tau$ & Log-slope of $\|\epsilon_n\|$ \\
\midrule
AE-only & 0.0432 $\pm$ 0.0038 & 0.1151 $\pm$ 0.0068 & 87.6 $\pm$ 1.9 & -- \\
\midrule
AE+ODE (base)  & 0.0431 $\pm$ 0.0039 & 0.1140 $\pm$ 0.0069 & 78.7 $\pm$ 26.6 & 0.0316 $\pm$ 0.0020 \\
AE+ODE (mixed) & 0.0435 $\pm$ 0.0038 & 0.1143 $\pm$ 0.0064 & 78.6 $\pm$ 26.6 & 0.0316 $\pm$ 0.0020 \\
AE+ODE (mix2)  & 0.0437 $\pm$ 0.0041 & 0.1142 $\pm$ 0.0068 & 78.5 $\pm$ 26.6 & 0.0315 $\pm$ 0.0020 \\
\bottomrule
\end{tabular} 
\end{table}

Local discrepancy injection and latent rollout deviation remain similar across the three exposure strategies. Nearly identical mean $\|\delta_n\|$ and mean $\|\epsilon_n\|$ values for Base, Mixed, and Mix2 appear in the Jacobian statistics, reported together with the Gray--Scott results in Table~\ref{tab:jacobian_summary}. Changing rollout exposure therefore does not materially alter transport sensitivity in this regime of subdued amplification.

During extrapolation, linear dimensionality reduction methods fail to maintain predictive accuracy (Table~\ref{tab:2d_representation}). Increasing the latent dimension from 16 to 32 does not improve POD/DMD performance, whereas the nonlinear autoencoder representation achieves substantially lower error at dimension 16.

\begin{table}[t]
\centering
\caption{
2D Burgers representation comparison across latent dimensions:
Physical-space relative $L^2$ error (RelL2), temporal correlation ($H_\tau$), and empirical growth rate (log-slope).
}
\label{tab:2d_representation}
\begin{tabular}{c|ccc|ccc}
\toprule
& \multicolumn{3}{c|}{AE (nonlinear representation)}
& \multicolumn{3}{c}{POD/DMD (linear projection)} \\
\cmidrule(r){2-4} \cmidrule(l){5-7}
Latent dim & Test RelL2 & $H_\tau$ & Log-slope
            & Test RelL2 & $H_\tau$ & Log-slope \\
\midrule
8  & 0.1127 & 88 & 0.0331 & 3.29340 & 34 & -0.01750 \\
16 & 0.1101 & 89 & 0.0336 & 2.43420 & 31 & -0.05700 \\
32 & 0.1170 & 87 & 0.0324 & 2.43420 & 31 & -0.05700 \\
\bottomrule
\end{tabular}
\end{table}

\subsection{Gray--Scott System: Strong Amplification Regime}
\label{sec:grayscott}

Under the boundary-initialized evaluation protocol, autonomous rollout is initialized from the final training latent state, and prediction errors accumulate progressively during autonomous rollout, indicating a transport regime in which locally injected discrepancies are progressively amplified during rollout.

Reaction--diffusion structures deteriorate much more rapidly under the Base strategy, which gradually loses the characteristic patterns and eventually produces smeared concentration fields (Fig.~\ref{fig:grayscott_main}). Mixed preserves the global structures for a substantially longer interval, whereas Gradient Clipping suppresses part of the instability but introduces localized artifacts. The large multi-step transport sensitivity ($S_{n,k}\gg1$) confirms that latent perturbations are strongly amplified during rollout; combined with the local discrepancy injection at each step, this amplification is consistent with the transport-based decomposition developed in Eq.~(16), providing a transport-based interpretation of why even small differences in training strategy can produce markedly different solution quality.

\begin{figure}[t]
\centering
\safeincludegraphics[width=\textwidth]{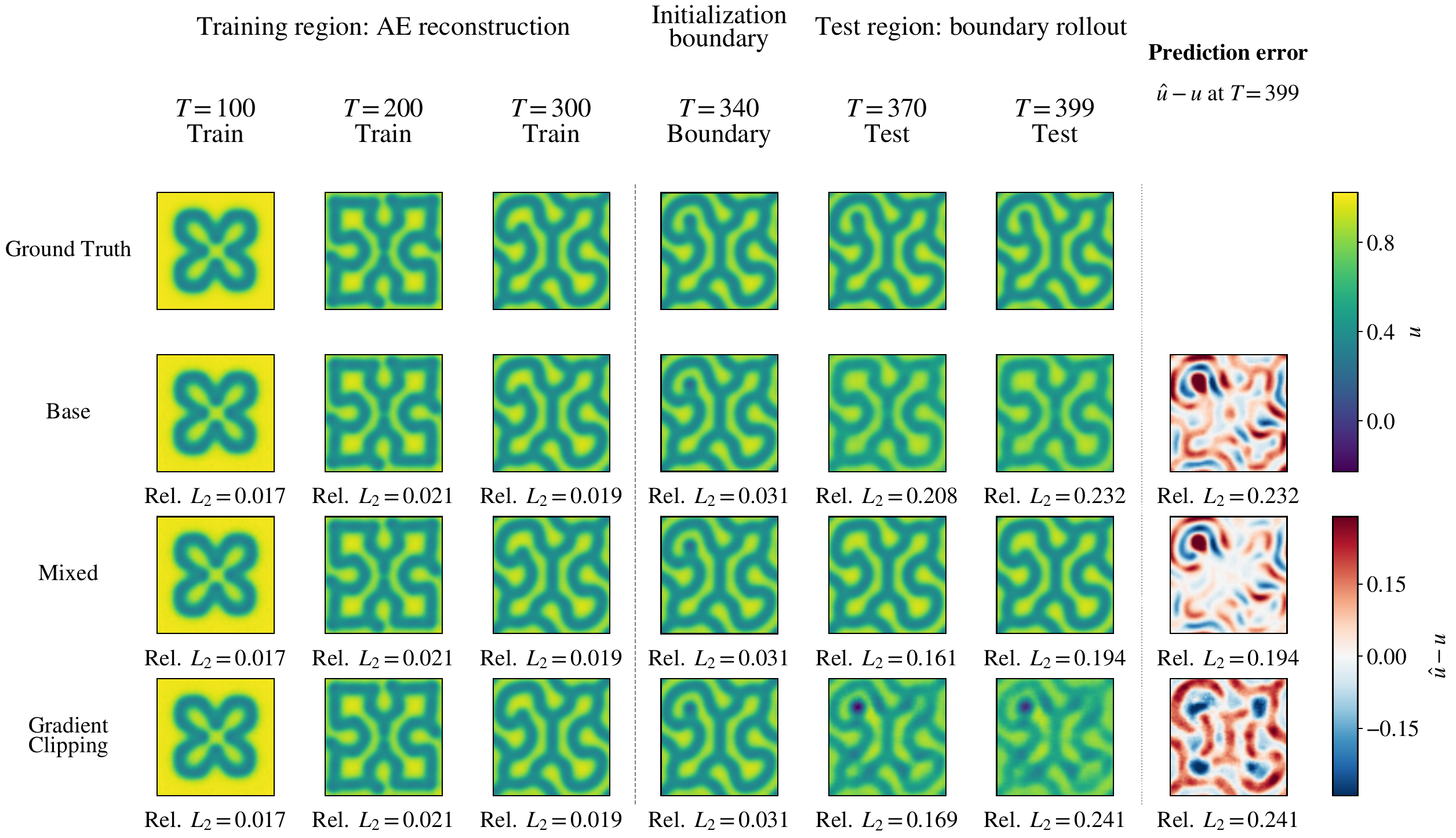}
\caption{
Representative Gray--Scott solution snapshots under boundary-initialized strict extrapolation.
Columns correspond to representative times spanning the training region, boundary initialization, and test region.
Rows compare the ground truth with decoded predictions produced by the Base, Mixed, and Gradient Clipping models.
}
\label{fig:grayscott_main}
\end{figure}

Across ten matched seeds, quantitative results confirm that Mixed achieves the lowest mean relative $L^2$ error (0.159) together with the smallest variability (standard deviation 0.016, Table~\ref{tab:grayscott_summary}). Gradient Clipping attains a comparable mean error (0.165), although its performance varies considerably across seeds, with errors ranging from 0.069 to 0.474. Statistical evidence separating Mixed and Gradient Clipping is limited for the current sample size.

\begin{table}[t]
\centering
\caption{
Gray--Scott boundary-initialized strict extrapolation results across 10 matched random seeds.
Autonomous rollout is initiated precisely at the train/test temporal boundary ($N_{\rm train}=340$).
}
\label{tab:grayscott_summary}
\begin{tabular}{lccccc}
\toprule
Method & Mean RelL2 & Median RelL2 & Std. & Min & Max \\
\midrule
Base & 0.340 & 0.279 & 0.189 & 0.164 & 0.678 \\
Mixed & 0.159 & 0.157 & 0.016 & 0.134 & 0.185 \\
Gradient Clipping & 0.165 & 0.120 & 0.125 & 0.069 & 0.474 \\
\bottomrule
\end{tabular}
\end{table}

Larger latent rollout deviations are generally accompanied by larger physical-space errors (Fig.~\ref{fig:grayscott_trajectory}). Relative to Base, Mixed maintains both quantities at consistently lower levels throughout rollout. Gradient Clipping reduces the latent rollout deviation, yet this reduction does not translate into a proportional improvement in physical-space accuracy. This observation is consistent with Eq.~(16): the latent rollout deviation reflects not only the transport amplification of perturbations but also the magnitude of the local discrepancy injection at each step, so the relationship between latent and physical errors is not governed by transport sensitivity alone.

\begin{figure}[t]
\centering
\safeincludegraphics[width=\textwidth]{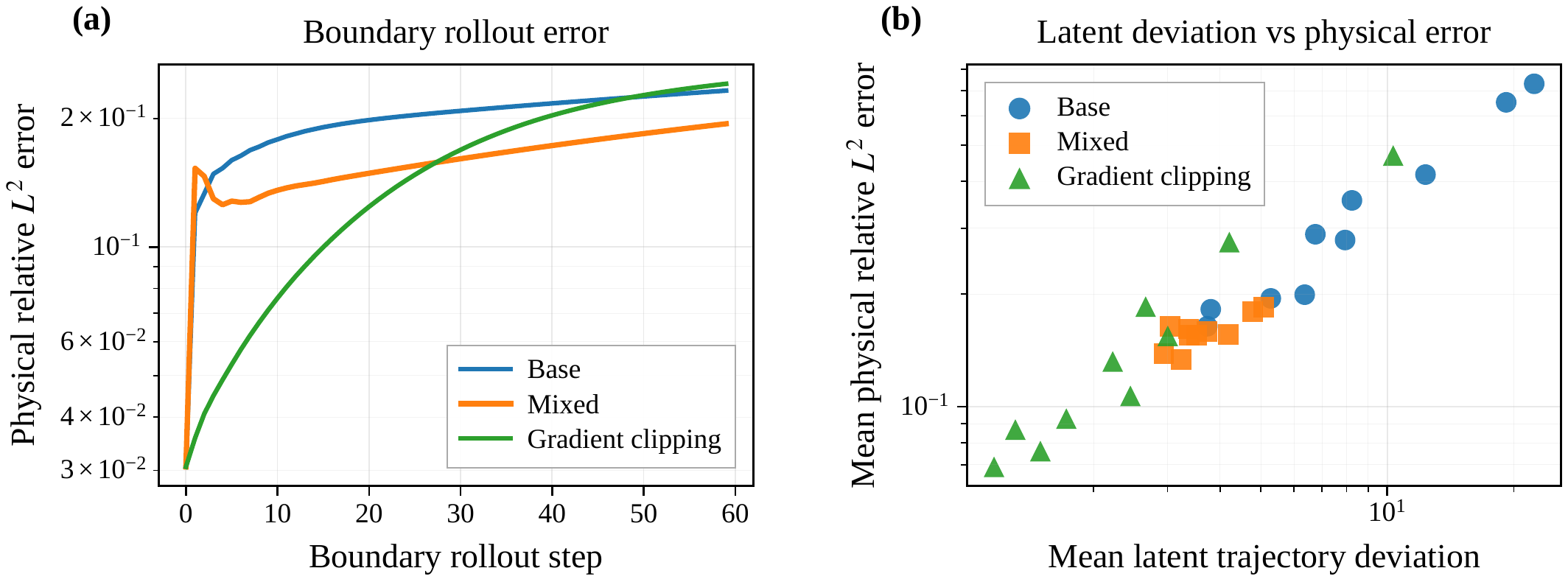}
\caption{
Interplay between physical extrapolation accuracy and latent rollout deviation.
(a) Temporal evolution of the physical-space relative $L^2$ error during autonomous boundary rollout.
(b) Cross-seed correlation between the mean latent rollout deviation and the resulting mean physical rollout error.
}
\label{fig:grayscott_trajectory}
\end{figure}

These effects are disentangled by the transport diagnostics (Fig.~\ref{fig:grayscott_jacobian}). At each step, the local discrepancy injection $\|\delta_n\|$ quantifies the error injected into the latent dynamics, whereas the multi-step transport sensitivity $\log_{10}S_k$ characterizes how these perturbations evolve during rollout. Together, the two quantities are consistent with the transport-based decomposition developed in Eq.~(16), which provides a transport-based interpretation of why similar local discrepancy injections can ultimately lead to substantially different long-horizon prediction errors: the amplification factor $S_{n,k}$ characterizes how strongly a given injection is transported and amplified during rollout, while the injection itself sets the initial magnitude of the perturbation that is subsequently transported.

\begin{figure}[t]
\centering
\safeincludegraphics[width=\textwidth]{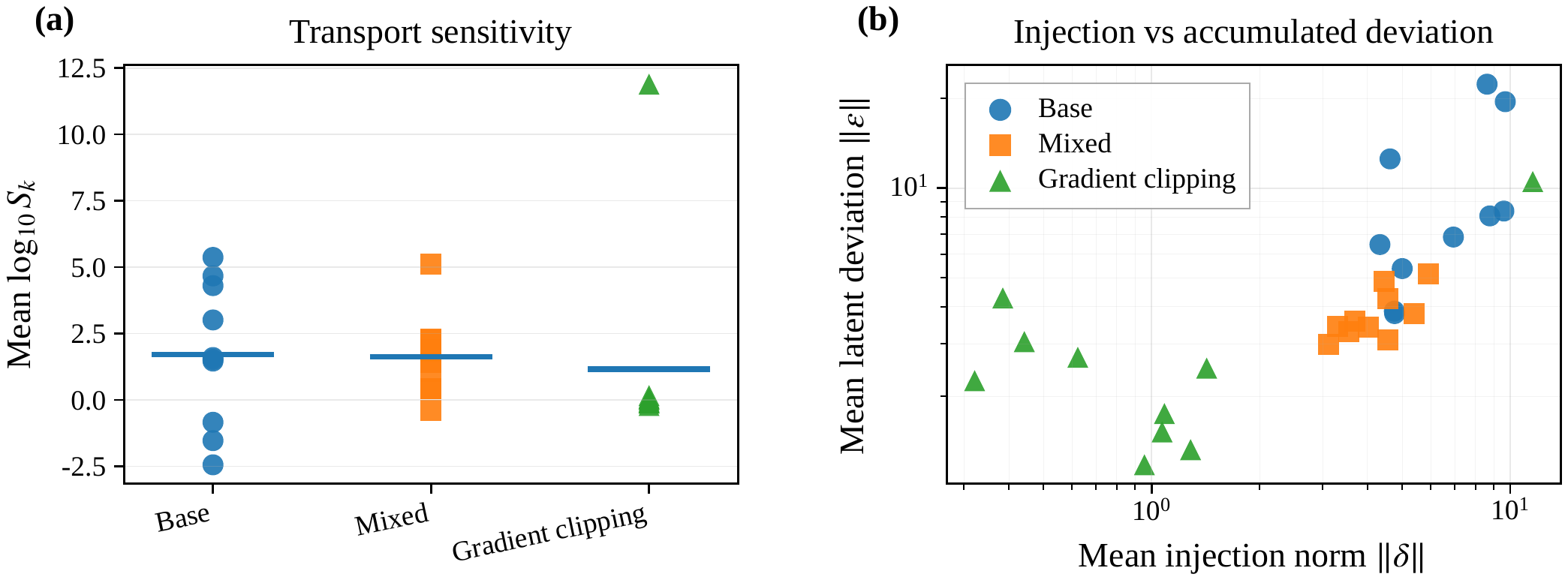}
\caption{
Jacobian-driven transport diagnostics under boundary-initialized extrapolation.
(a) Multi-step transport sensitivity profiles tracked via $\log_{10} S_k$.
(b) Scatter profile mapping the relation between local discrepancy injection and the latent rollout deviation.
}
\label{fig:grayscott_jacobian}
\end{figure}

\subsection{Transport Diagnostic Analysis}
\label{sec:transport_diagnostics}

Jacobian metrics of the learned rollout map, averaged over ten seeds, are listed in Table~\ref{tab:jacobian_summary}. Mean $\log_{10}(S_{n,k})$ is 0.006 for 2D Burgers models, compared to 1.713 for Base, 1.628 for Mixed, and 1.163 for Gradient Clipping in Gray--Scott. As a transport sensitivity diagnostic, larger values of $S_{n,k}$ indicate stronger amplification of latent perturbations, which contributes to larger latent rollout deviations $\|\epsilon_n\|$.

\begin{table}[t]
\centering
\caption{
Jacobian diagnostics for boundary-initialized rollouts averaged over ten matched seeds. 
The mean of $\log_{10}(S_{n,k})$ quantifies transport sensitivity, 
$\|\delta_n\|$ measures the local discrepancy injection, 
$\|\epsilon_n\|$ denotes the latent rollout deviation, 
and $R^2$ evaluates the correlation between transport amplification and latent rollout deviation.
}
\label{tab:jacobian_summary}
\begin{tabular}{llcccc}
\toprule
System & Method
& Mean $\log_{10}(S_{n,k})$
& Mean $\|\delta_n\|$
& Mean $\|\epsilon_n\|$
& $R^2$
\\
\midrule

2D Burgers
& Base
& 0.006
& 0.024
& 0.159
& 0.891
\\
& Mixed
& 0.006
& 0.024
& 0.158
& 0.892
\\
& Mix2
& 0.006
& 0.023
& 0.151
& 0.892
\\

\midrule

Gray--Scott
& Base
& 1.713
& 6.708
& 9.716
& 0.760
\\
& Mixed
& 1.628
& 4.254
& 3.788
& 0.816
\\
& Clipped
& 1.163
& 1.915
& 3.109
& 0.736
\\

\bottomrule
\end{tabular}
\end{table}

For Gray--Scott, Mixed achieves lower physical rollout error (0.159) than Base (0.340) despite similar transport sensitivity ($\log_{10}(S_{n,k}) = 1.628$ versus $1.713$). This observation directly illustrates the decomposition introduced in Eq.~(16): long-term rollout error is governed by the combined effect of transport amplification, captured by $S_{n,k}$, and local discrepancy injection, captured by $\|\delta_n\|$. Mixed reduces $\|\delta_n\|$ from 6.708 (Base) to 4.254 while operating under a comparable amplification regime, and this reduction in local discrepancy injection accompanies the lower latent rollout deviation $\|\epsilon_n\|$ from 9.716 (Base) to 3.788 (Mixed). The experimental results therefore support the interpretation that neither transport amplification nor local discrepancy injection alone determines long-term rollout error; rather, their joint contribution, as expressed by Eq.~(16), shapes the observed behavior. This observation is consistent with the theoretical decomposition of Proposition 2.
Across all three Gray--Scott models, the corresponding $R^2$ values remain moderately high, with Mixed attaining the largest value ($0.816$ versus $0.760$ for Base and $0.736$ for Clipped), indicating that transport sensitivity and latent rollout deviation retain a stable statistical association across the evaluated seeds. This association, however, does not imply that transport sensitivity alone dictates prediction accuracy; the latent rollout deviation also depends on the local discrepancy injection at each step, consistent with the combined formulation in Eq.~(16).

\subsection{Statistical Validation}
\label{sec:stats}

Matched seeds were evaluated using paired $t$-tests and Wilcoxon signed-rank tests (Table~\ref{tab:stats}).

For 1D and 2D Burgers, comparisons across exposure strategies yield $p > 0.20$ for all paired tests. In Gray--Scott, Gradient Clipping differs significantly from Base under both tests ($p = 0.018$ $t$-test, $p = 0.002$ Wilcoxon). For the Mixed versus Gradient Clipping comparison, the Wilcoxon signed-rank test detects a significant difference ($p = 0.020$), whereas the paired $t$-test does not reach the conventional 0.05 significance level ($p = 0.063$); the evidence separating Mixed and Gradient Clipping is therefore test-dependent and should be interpreted cautiously.

\begin{table}[t]
\centering
\caption{
Paired significance tests evaluated across matched random seeds.
}
\label{tab:stats}
\begin{tabular}{lccc}
\toprule
System & Comparison & Paired t-test & Wilcoxon \\
\midrule

1D Burgers
& Base vs Mixed
& 0.209
& 0.695
\\

2D Burgers
& Base vs Mixed
& 0.450
& 1.000
\\

2D Burgers
& Base vs Mix2
& 0.458
& 0.846
\\

Gray--Scott
& Mixed vs Clipped
& 0.063
& 0.020
\\

Gray--Scott
& Base vs Clipped
& 0.018
& 0.002
\\

\bottomrule
\end{tabular}
\end{table}

In strongly amplifying dynamics such as Gray--Scott, the influence of optimization strategy on rollout accuracy becomes more apparent, although the statistical evidence for some pairwise comparisons remains test-dependent. In contrast, seed variation dominates in systems with mild transport sensitivity such as the Burgers equations.

\section{Conclusion}
\label{sec:conclusion}

The proposed path-integral formulation separates local discrepancy injection from transport effects in learned latent dynamics. In weakly amplifying regimes such as Burgers systems, rollout performance is governed by local reconstruction accuracy. In strongly amplifying regimes such as the Gray--Scott system, transport amplification dominates long-horizon state drift, requiring explicit trajectory regularization. Incorporating manifold-aware regularization to constrain off-manifold latent solver trajectories remains a topic for future work.


\bibliographystyle{plain}
\bibliography{refs}

\end{document}